\RequirePackage{fix-cm}
\documentclass[envcountsect,envcountsame]{svjour3}
\usepackage{amsmath,amssymb,mathtools,subcaption,algorithm,bm,bbm,xcolor}
\usepackage{latexsym,eqnarray}
\usepackage{booktabs}
\usepackage[colorlinks=true,citecolor=blue,linkcolor=blue,urlcolor=blue,bookmarks,bookmarksopen,bookmarksdepth=2,backref=page]{hyperref}

\usepackage{bookmark}
\bookmarksetup{
	numbered, 
	open
}

\smartqed

\newcommand*{\doi}[1]{\href{\detokenize{#1}}{doi: \detokenize{#1}}}
\renewcommand*{\backref}[1]{}
\renewcommand*{\backrefalt}[4]{%
	\ifcase #1 (Not cited.)%
	\or        (Cited on page~#2.)%
	\else      (Cited on pages~#2.)%
	\fi}

\def\cM{{\mathcal{M}}}
\newcommand{\B}{\mathcal{B}}
\def \F {{\mathbb F}}
\def \N {{\mathbb N}}

\newtheorem{openproblem}{Open problem}
\usepackage[T1]{fontenc}

\usepackage[labelsep=period]{caption}
\usepackage{chngcntr}
\counterwithin*{figure}{section}

\renewcommand{\thefigure}{\arabic{section}.\arabic{figure}} 

\numberwithin{equation}{section}

\begin{document}
\title{Bent functions satisfying the dual bent condition and permutations with the $(\mathcal{A}_m)$ property
\thanks{The first version of this work~\cite{PPKZ_BFA23} was presented in the ``Eighth International Workshop on Boolean Functions and their Applications (BFA 2023)''.}
}

\titlerunning{Bent functions satisfying the dual bent condition}

\author{Alexandr Polujan \and  Enes Pasalic \and Sadmir Kudin \and Fengrong Zhang}

\authorrunning{A. Polujan, E. Pasalic, S. Kudin, F. Zhang}

\institute{
			Alexandr Polujan \at
			Otto von Guericke University, Universit\"{a}tsplatz 2, 39106, Magdeburg, Germany \\
			\email{alexandr.polujan@ovgu.de}
			\and
            Enes Pasalic \at 
            University of Primorska, FAMNIT \& IAM, Glagolja\v{s}ka 8, 6000 Koper, Slovenia \\
            \email{enes.pasalic6@gmail.com}
            \and
            Sadmir Kudin \at 
            University of Primorska, FAMNIT \& IAM, Glagolja\v{s}ka 8, 6000 Koper, Slovenia \\
            \email{sadmir.kudin@iam.upr.si}
            \and 
           Fengrong Zhang \at
              State Key Laboratory of Integrated Services Networks, Xidian University, Xian 710071, P.R. China\\
              \email{zhfl203@163.com}              
}

\date{Received: date / Accepted: date}

\maketitle

\begin{abstract}
The concatenation of four Boolean bent functions $f=f_1||f_2||f_3||f_4$ is bent if and only if the dual bent condition  $f_1^* + f_2^* + f_3^* + f_4^* =1$ is satisfied. However, to specify four bent  functions satisfying this duality condition is in general quite a  difficult task. Commonly, to simplify this  problem, certain connections between $f_i$ are assumed,  as well as functions $f_i$ of a special shape are considered, e.g., $f_i(x,y)=x\cdot\pi_i(y)+h_i(y)$ are Maiorana-McFarland bent functions. In the case when permutations $\pi_i$ of $\F_2^m$ have the $(\mathcal{A}_m)$ property and Maiorana-McFarland bent functions $f_i$ satisfy the additional condition $f_1+f_2+f_3+f_4=0$, the dual bent condition is known to have a relatively simple shape allowing to specify the functions $f_i$ explicitly. In this paper, we generalize this result for the case when Maiorana-McFarland bent functions $f_i$ satisfy the condition $f_1(x,y)+f_2(x,y)+f_3(x,y)+f_4(x,y)=s(y)$ and provide a construction of new permutations with the $(\mathcal{A}_m)$ property from the old ones. Combining these two results, we obtain a recursive construction method of bent functions satisfying the dual bent condition. Moreover, we provide a generic condition on the Maiorana-McFarland bent functions $f_1,f_2,f_3,f_4$ stemming from the permutations of $\F_2^m$ with the $(\mathcal{A}_m)$ property, such that the concatenation $f=f_1||f_2||f_3||f_4$ does not belong, up to equivalence, to the Maiorana-McFarland class. Using monomial permutations $\pi_i$ of $\F_{2^m}$ with the $(\mathcal{A}_m)$ property and monomial functions $h_i$ on $\F_{2^m}$, we provide explicit constructions of such bent functions; a particular case of our result shows how one can construct bent functions from APN permutations, when $m$ is odd. Finally, with our construction method, we explain how one can construct homogeneous cubic bent functions, noticing that only very few design methods of these objects are known.
\end{abstract}
\section{Introduction}
Bent functions are fascinating combinatorial objects that has been introduced by Rothaus in~\cite{Rot} and since then have been attracting a lot of attention from the research community for the past four decades~\cite{FourDecades}. There exists many constructions of bent functions (we refer to~\cite{Mesnager}, for comprehensive surveys and extensive references on the subject), which could be divided into \textit{primary constructions}, e.g., Maiorana-McFarland $\mathcal{M}$ and Partial Spread $\mathcal{PS}$~\cite{Dillon,Rot}, $\mathcal{C}$ and $\mathcal{D}$ classes of Carlet~\cite{CC93}, to name a few, as well as \textit{secondary constructions}, which employ the already known bent functions to construct new ones~\cite{Carlet_Secondary,Hodzic2020,Rot,ZWPC_Secondary}.

A well-known secondary construction of bent functions is the so-called \textit{bent 4-concatenation}:  given four bent functions $f_1,f_2,f_3,f_4$ in $n$ variables, one may try to concatenate their truth tables such that the resulting function $f$ is a bent function in $n+2$ variables. The first systematic study of this secondary construction was given in~\cite{Decom}, where the authors investigated under which conditions a given bent function $f$ in $n+2$ variables can be written as a concatenation of four bent functions $f_1,f_2,f_3,f_4$ in $n$ variables; in~\cite{Decom}, such a representation was called a bent \textit{4-decomposition}. Notably, in~\cite[Theorem 7]{Decom}, there was provided a criterion for a bent function $f$ to admit a bent 4-decomposition in terms of the second-order derivatives of the dual bent function $f^*$. While the analysis of bent 4-decompositions is useful for distinguishing different functions, the reverse process can be used for the construction of new families of bent functions. In~\cite{HPZ2019}, it was shown that for a concatenation of four bent functions $f=f_1||f_2||f_3||f_4$, the necessary and sufficient condition that $f$ is bent is  that the \textit{dual bent condition} is satisfied, i.e., $f_1^*+  f_2^* + f_3^*   +  f_4^* =1$, see~\cite[Theorem III.1]{HPZ2019}; note that this result (but in different formulations) was also obtained in~\cite{Preneel91,Tokareva2011AMC}. In~\cite{Mesnager}, Mesnager provided generic secondary constructions of bent functions as well as several explicit infinite families of bent functions and their duals using the permutations of $\F_2^m$ with the $(\mathcal{A}_m)$ property; the latter was formally introduced later in~\cite{Mesnager2014Am} as follows. Permutations  $\pi_1,\pi_2,\pi_3$ satisfy the $(\mathcal{A}_m)$ property if $\pi=\pi_1+\pi_2 + \pi_3$  is  a permutation and its inverse is given by $\pi^{-1}=\pi_1^{-1} + \pi_2^{-1} + \pi_3^{-1}$. Despite the fact, that in general it is difficult to determine explicitly permutations satisfying the strong condition~$(\mathcal{A}_m)$, new classes of such mappings were constructed in~\cite{Mesnager2014Am} and also used to provide more constructions of bent functions using involutions~\cite{Mesnager2015}.

While finding bent functions satisfying the dual bent condition and constructing permutations with the $(\mathcal{A}_m)$ property are two difficult independent open problems (as indicated in~\cite{SHCF,Mesnager2014Am,PPKZ2023}), their consideration ``as a whole'' was shown to be a potential source for the new secondary constructions of bent functions~\cite{CCDS2019}.

The main aim of this paper is three-fold. Firstly, to develop further the construction techniques of permutations with the $(\mathcal{A}_m)$ property by extending the results in~\cite{CCDS2019}. Secondly, based on the obtained permutations $\pi_i$ with the $(\mathcal{A}_m)$ property, to specify Maiorana-McFarland bent functions $f_i(x,y)=x\cdot\pi_i(y)+h_i(y)$ which satisfy the dual bent condition, and, hence, whose bent 4-concatenation gives rise to new bent functions. Finally, to show that the obtained constructions of bent functions are indeed ``new'', in the sense that the obtained constructions are different from the building blocks used in the concatenation (formally it means that $f=f_1||f_2||f_3||f_4$ is outside the completed Maiorana-McFarland class $\cM^\#$). 

The rest of the paper is organized in the following way. In Subsection~\ref{subsec: 1.1 Prelim}, we give basic definitions related to Boolean functions, and in Subsection~\ref{subsec: 1.2 Dual bent and Am} we summarize important definitions and results regarding the dual bent condition and permutations with the $(\mathcal{A}_m)$ property. In Section~\ref{sec: 2 Secondary constructions}, we provide secondary constructions of permutations with the $(\mathcal{A}_m)$ property and use them to construct bent functions by concatenating the corresponding Maiorana-McFarland bent functions satisfying the dual bent condition. To be more precise, in Subsection~\ref{subsec: 2.2 Permutations with Am new from old}, we provide a secondary construction of permutations with the $(\mathcal{A}_m)$ property using the piecewise permutations. In Subsection~\ref{subsec: 2.1 Concatenating MM generalized}, we first generalize~\cite[Theorem 7]{CCDS2019} (see Theorem~\ref{th: Frobenius extended}), which gives a possibility to concatenate larger classes of Maiorana-McFarland bent functions and to get in such a way new bent functions.

In the remaining sections, we investigate, under which conditions concatenation of Maiorana-McFarland bent functions is again bent and does not belong, up to equivalence, to the Maiorana-McFarland class. In Section~\ref{sec: 3}, we provide more conditions on permutations $\pi_i$ such that the concatenation  $f=f_1||f_2||f_3||f_4$ of Maiorana-McFarland bent functions  $f_i(x,y)=x\cdot\pi_i(y)+h_i(y)$ is outside the completed Maiorana-McFarland class $\cM^\#$. In Section~\ref{sec: 4}, we provide explicit construction methods so that the concatenation of four  Maiorana-McFarland bent functions on $\F_2^n$ (using suitable permutation monomials) generate bent functions on $\F_2^{n+2}$ outside $\cM^\#$. 
In this way, we provide a solution to~\cite[Open Problem 5.16]{PPKZ2023}. In Section~\ref{sec: 5 Hom bent}, we show that it is possible to construct homogeneous cubic bent functions using our construction methods (note that only a few construction methods of such functions are known in the literature) but due to the underlying difficulty of this problem further research efforts are required. Finally, we conclude the paper in Section~\ref{sec: 6 Conclusion}.

\subsection{Preliminaries}\label{subsec: 1.1 Prelim}
Let $\mathbb{F}_2^n$ be the vector space of all binary $n$-tuples $x=(x_1,\ldots,x_n)$, where $x_i \in \mathbb{F}_2$. For two elements $x=(x_1,\ldots,x_n)$ and $y=(y_1,\ldots,y_n)$ of $\mathbb{F}^n_2$, we define the scalar product over $\mathbb{F}_2$ as $x\cdot y=x_1 y_1 + \cdots +  x_n y_n$. For $x=(x_1,\ldots,x_n)\in \mathbb{F}^n_2$, the Hamming weight of $x$ is defined by $wt(x)=\sum^n_{i=1} x_i$. Throughout the paper, we denote by $0_n=(0,0,\ldots,0)\in \mathbb{F}^n_2$ the all-zero vector with $n$ coordinates. If necessary, we endow $\F_2^n$ with the structure of the finite field $\left(\F_{2^{n}},+,\cdot \right)$. An element $\alpha \in \mathbb{F}_{2^n}$ is called a \emph{primitive element}, if it is a generator of the multiplicative group $\mathbb{F}_{2^n}^*$. The \emph{absolute trace} $Tr\colon \mathbb{F}_{2^n} \rightarrow \mathbb{F}_{2}$ is given by $Tr(x) =\sum_{i=0}^{n-1} x^{2^{i}}$.

In this paper, we denote the set of all Boolean functions in $n$ variables by $\mathcal{B}_n$. One can uniquely represent any Boolean function $f\in\mathcal{B}_n$ using the \textit{algebraic normal form (ANF, for short)}, which is given by $f(x_1,\ldots,x_n)=\sum_{u\in \mathbb{F}^n_2}{\lambda_u}{(\prod_{i=1}^n{x_i}^{u_i})}$, where $x_i, \lambda_u \in \mathbb{F}_2$ and $u=(u_1, \ldots,u_n)\in \mathbb{F}^n_2$.
	The \textit{algebraic degree} of $f\in\mathcal{B}_n$, denoted by $\deg(f)$, is the maximum Hamming weight of $u \in \F_2^n$ for which $\lambda_u \neq 0$ in its ANF. A Boolean function $f\in\mathcal{B}_n$ is called \textit{homogeneous} if all the monomials in its ANF have the same algebraic degree.
	
	The \textit{first order-derivative} of a function $f\in\mathcal{B}_n$ in the direction $a \in \F_2^n$ is the mapping $D_{a}f(x)=f(x+a) +  f(x)$. Derivatives of higher orders are defined recursively, i.e., the \emph{$k$-th order derivative} of a function $f\in\mathcal{B}_n$ is defined by $D_Vf(x)=D_{a_k}D_{a_{k-1}}\ldots D_{a_1}f(x)=D_{a_k}(D_{a_{k-1}}\ldots D_{a_1}f)(x)$, where $V=\langle a_1,\ldots,a_k \rangle$ is a vector subspace of $\F_2^n$ spanned by elements $a_1,\ldots,a_k\in\F_2^n$. An element $a\in\F_2^n$ is called a \textit{linear structure} of $f\in\mathcal{B}_n$, if $f(x+a)+f(x)=const\in \F_2$ for all $x\in\F_2^n$. We say that  $f\in\mathcal{B}_n$ \textit{has no linear structures}, if $0_n$ is the only linear structure of $f$. 
	
	The \emph{Walsh-Hadamard transform} (WHT) of $f\in\mathcal{B}_n$, and its inverse WHT, at any point $a\in\mathbb{F}^n_2$ are defined, respectively, by
	\begin{equation*}
		W_{f}(a)=\sum_{x\in \mathbb{F}_2^n}(-1)^{f(x) +  a\cdot x} \quad\mbox{and}\quad
		(-1)^{f(x)}=2^{-n}\sum_{a\in \mathbb{F}_2^n}W_f(a)(-1)^{a\cdot x}.
	\end{equation*}
	
	For even $n$, a function $f\in\mathcal{B}_n$ is called {\em bent} if $W_f(u)=\pm2^{\frac{n}{2}}$ for all $u\in\F_2^n$. For a bent function $f\in\mathcal{B}_n$, a Boolean function $f^*\in \mathcal{B}_n$ defined by $W_f(u)=2^{\frac{n}{2}}(-1)^{f^*(u)}$ for all $u\in\F_2^n$ is a bent function, called the {\it dual} of $f$. 
 Two Boolean functions $f,f'\in\mathcal{B}_n$ are called \textit{(extended-affine) equivalent}, if there exists an affine permutation $A$ of $\F_2^n$ and an affine function $l\in\mathcal{B}_n$, such that $f\circ A + l= f'$. It is well-known, that extended-affine equivalence preserves the bent property. A class of bent functions $\mathit{B}_n \subset \mathcal{B}_n$ is \emph{complete} if it is globally invariant under extended-affine equivalence. By $\mathit{B}_n^\#$ we denote the \emph{completed  $\mathit{B}_n$ class}, which is the smallest possible complete class that  contains $\mathit{B}_n$.
	
		The \textit{Maiorana-McFarland class} $\cM$ is the set of $n$-variable ($n=2m$) Boolean bent functions of the form
	\[
	f(x,y)=x \cdot \pi(y)+ h(y), \mbox{ for all } x, y\in\F_2^m,
	\]
	where $\pi$ is a permutation on $\F_2^m$, and $h$ is an arbitrary Boolean function on
	$\F_2^m$. 
	Using  Dillon's criterion below (the proof of this statement can be found in~\cite{Carlet2021} or \cite[p. 19]{Polujan2020}), one can show that a given Boolean bent function $f\in\mathcal{B}_n$ does (not) belong to the completed Maiorana-McFarland class $\mathcal{M}^\#$.
	\begin{lemma} \cite[p. 102]{Dillon}\label{lem M-M second}
		Let $n=2m$. A Boolean bent function $f\in\mathcal{B}_n$ belongs to $\cM^{\#}$ if and only if
		there exists an $m$-dimensional linear subspace $U$ of $\F_2^n$ such that the second-order derivatives
		$ D_{a}D_{b}f(x)=f(x) +  f(x +  a) +  f(x +  b) +  f(x +  a +  b)$
		vanish for any $ a,  b \in U$.
	\end{lemma} 
	Following the terminology in~\cite{PPKZ2023,Polujan2020}, for a  function $f\in\mathcal{B}_n$, we call a vector subspace $U$ such that $D_a D_b f=0$ for all $a,b\in U$ an \textit{$\mathcal{M}$-subspace} of $f$. Note that if $f\in\mathcal{M}$, then at least one $\mathcal{M}$-subspace of $f$ has the form $U=\F_2^m \times \{0_m\}$, which we call the \textit{canonical $\mathcal{M}$-subspace} of $f$. However, if $f\in\cM^\#$ then this is not true in general.
	
	
\subsection{The dual bent condition and the $(\mathcal{A}_m)$ property}\label{subsec: 1.2 Dual bent and Am}
In this subsection, we introduce the main terminology and summarize several important results regarding the bent 4-concatenation, the dual bent condition and permutations of $\F_2^m$ with the $(\mathcal{A}_m)$ property.

Throughout the paper,  we denote the \textit{(canonical) concatenation} of four functions $f_i \in \B_n$  as  $f=f_1||f_2||f_3||f_4 \in \B_{n+2}$, whose ANF is given by
\begin{equation}\label{eq:ANF_4conc}
	\begin{split}
	f(z,z_{n+1},z_{n+2})=&f_1(z) +  z_{n+1}(f_1 +  f_3)(z) +  z_{n+2}(f_1 +  f_2)(z)\\ + &  z_{n+1}z_{n+2}(f_1 +  f_2 +  f_3 +  f_4)(z).
	\end{split}
\end{equation}
Notice that the subfunctions $f_i$ of $f$ are defined on the four cosets of $\F_2^n$ so that $f_1(z)=f(z,0,0), \ldots,  f_4(z)=f(z,1,1)$.
Conversely, as shown in \cite{Decom}, any bent function on $\F_2^{n+2}$ can be decomposed  into four subfunctions  $f_i \in \F_2^n$, where all $f_i$ are bent, disjoint spectra semi-bent functions or suitable 5-valued spectra functions. The latter two cases have been recently analyzed in \cite{Bent_Decomp2022}, where efficient methods of designing bent functions outside $\cM^\#$ were proposed. 
 For a concatenation of four bent functions $f=f_1||f_2||f_3||f_4$, the necessary and sufficient condition that $f$ is bent is  that the \textit{dual bent condition} is satisfied~\cite[Theorem III.1]{HPZ2019}, i.e., $f_1^*+  f_2^* + f_3^*   +  f_4^* =1$. For our purpose, one  construction method of bent functions satisfying the dual bent condition given in Theorem \ref{th:cond_on_h} below, using the permutations of $\F_2^m$ with the $(\mathcal{A}_m)$ property (defined below), is of particular importance.

\begin{definition}\cite{Mesnager2014_Several_New}
	Let $\pi_1,\pi_2,\pi_3$ be three permutations of $\F_2^m$. We say that $\pi_1,\pi_2,\pi_3$ have the $(\mathcal{A}_m)$ property if
	\begin{enumerate}
		\item $\pi_4=\pi_1+\pi_2 + \pi_3$  is  a permutation and
		\item $\pi_4^{-1}=\pi_1^{-1} + \pi_2^{-1} + \pi_3^{-1} $.
	\end{enumerate} 
\end{definition}
As the following result shows, the dual bent condition could be satisfied~\cite{CCDS2019} by using Maiorana-McFarland bent functions arising from permutations with the $(\mathcal{A}_m)$ property~\cite{Mesnager2014Am}.
\begin{theorem}\cite[Theorem 7]{CCDS2019}\label{th:cond_on_h}
	Let $ f_j (x, y) = Tr(x \pi_j (y)) + h_j (y)$ for $ j \in  \{1, 2, 3\}$ and $x, y \in
	\F_{2^{m}}$, where the permutations $\pi_j$ satisfy the condition ($\mathcal{A}_m$). If the functions $h_j$ satisfy
	\begin{equation}\label{eq:dualconditiononh}
		h_1(\pi^{-1}_1 (y))+h_2(\pi^{-1}_2 (y))+h_3(\pi^{-1}_3 (y))+(h_1+h_2+h_3)((\pi_1+\pi_2+\pi_3)^{-1}(y)) = 1, 
	\end{equation}
	then $f_1, f_2, f_3$ satisfy $f^*_1 + f^*_2 + f^*_3 + f^*_4=1$, where $f_1+f_2 +f_3=f_4$. Consequently, $f=f_1||f_2||f_3||f_4$ is bent.
\end{theorem}
Notice that $f_4$ is defined as $f_4=f_1+f_2+f_3$ and the  functions $h_i(y)$ must be carefully selected (assuming that $\pi_1,\pi_2,\pi_3$ satisfy the ($\mathcal{A}_m$) property) so that \eqref{eq:dualconditiononh} is satisfied.
With the following recent result, it is possible to show that in certain cases the bent 4-concatenation of Maiorana-McFarland bent functions gives a bent function outside $\cM^\#$.

\begin{theorem}\cite{PPKZ2023} \label{th:sharingcommonsubspace}Let $f_1, \ldots,f_4$ be four bent  functions on $\F_2^n$, with $n=2m$, satisfying the following conditions:
	\begin{enumerate}
		\item $f_1, \ldots,f_4$ belong to $\cM^\#$ and  share a unique  $\mathcal{M}$-subspace of dimension $m$;
		\item  $f=f_1||f_2||f_3||f_4 \in\B_{n+2}$ is a bent function;
	\end{enumerate}
	Let $V$ be an $(\frac{n}{2}-1)$-dimensional subspace  of $\F_2^n$  such that $D_aD_bf_i=0$, for all $a,b \in V$; $i=1, \ldots ,4.$
	If for any $v\in  \F_2^n $ and  any such $V \subset \F_2^n$, there exist $u^{(1)},u^{(2)},u^{(3)}\in V $ such that the following three conditions hold simultaneously 
	\begin{enumerate}
		\item[1)] $D_{u^{(1)}}f_1(x)+D_{u^{(1)}}f_2(x+v)\neq 0,~\textit{or}~D_{u^{(1)}}f_3(x)+D_{u^{(1)}}f_4(x+v)\neq 0,$
		\item[2)] $D_{u^{(2)}}f_1(x)+D_{u^{(2)}}f_3(x+v)\neq 0,~\textit{or}~D_{u^{(2)}}f_2(x)+D_{u^{(2)}}f_4(x+v)\neq 0,$
		\item[3)] $D_{u^{(3)}}f_2(x)+D_{u^{(3)}}f_3(x+v)\neq 0,~\textit{or}~D_{u^{(3)}}f_1(x)+D_{u^{(3)}}f_4(x+v)\neq 0,$
	\end{enumerate}
	then $f$ is outside $\cM^\#$.
\end{theorem}

\section{Secondary constructions of bent functions satisfying the dual bent condition and permutations with the $(\mathcal{A}_m)$ property}
	\label{sec: 2 Secondary constructions}
	 Constructing new permutations of $\F_2^m$ satisfying  the $(\mathcal{A}_m)$ property  is a well-known challenging problem. In this section, we explain how one can construct permutations of $\F_2^{m+k}$ that satisfy the $(\mathcal{A}_{m+k})$ property from permutations of $\F_2^{m}$ with the $(\mathcal{A}_{m})$	property, where $k\in \N$ is arbitrary. Consequently, we use this construction method for specifying Maiorana-McFarland bent functions satisfying the dual bent condition. In this way, we obtain a secondary construction of bent functions.
	 \subsection{Permutations with the $(\mathcal{A}_m)$ property: New from old}\label{subsec: 2.2 Permutations with Am new from old}
	 Since our construction method of new permutations with the $(\mathcal{A}_m)$ property from old is based on the piecewise permutations, we need first to express the inverse of a permutation defined in a piecewise manner, which is required for checking the $(\mathcal{A}_m)$  property. Therefore, we begin this subsection with the following auxiliary result.
	 \begin{lemma}\label{eq: new permutation inverse}
	 	Let $\sigma_1,\sigma_2$ be permutations of $\F_2^m$. Define the mapping $\pi \colon \F_2^{m+1} \to \F_2^{m+1}$ by
	 	\begin{equation*}
	 		\pi(y,y_{m+1})= \left( y_{m+1}\sigma_1(y)+(1+y_{m+1}) \sigma_2(y) , y_{m+1} \right),
	 	\end{equation*}
	 	 for all $y \in \F_2^m, y_{m+1} \in \F_2$. Then, $\pi$ is a permutation, and its inverse on $\F_2^{m+1}$ is given by the permutation $\rho$ on $\F_2^{m+1}$, defined by
	 	\begin{equation*}
	 		\rho(y,y_{m+1})= \left( y_{m+1}\sigma_1^{-1}(y)+(1+y_{m+1}) \sigma_2^{-1}(y) , y_{m+1} \right), 
	 	\end{equation*}
 	for all  $y \in \F_2^m, y_{m+1} \in \F_2$.
	 \end{lemma}
	 \begin{proof}
	 	First, it is convenient to rewrite $\pi$ and $\rho$ in the piecewise form as follows:
	 	\begin{equation*}
	 		\begin{split}
	 			\pi(y,y_{m+1})=&\begin{cases}
	 				(\sigma_1(y),1) & \text{if } y_{m+1}=1\\
	 				(\sigma_2(y),0) & \text{if } y_{m+1}=0
	 			\end{cases}\quad\mbox{and}\\ \rho(y,y_{m+1})=&\begin{cases}
	 				(\sigma_1^{-1}(y),1) & \text{if } y_{m+1}=1\\
	 				(\sigma_2^{-1}(y),0) & \text{if } y_{m+1}=0
	 			\end{cases}.
	 		\end{split}
	 	\end{equation*}
	 	Since $\sigma_1$ and $\sigma_2$ are permutations of $\F_2^m$, so are  $\sigma_1^{-1}$ and $\sigma_2^{-1}$, and hence $\pi$ and $\rho$ are permutations of $\F_2^{m+1}$. Now, we show that $\rho$ is indeed the inverse of the permutation $\pi$. We compute $\pi(\rho(y,y_{m+1}))$, whose expression is given as follows
	 	\begin{equation}
	 		\begin{split}
	 			\pi(\rho(y,y_{m+1}))&=\begin{cases}
	 				\pi(\rho(y,1)) & \text{if } y_{m+1}=1\\
	 				\pi(\rho(y,0)) & \text{if } y_{m+1}=0
	 			\end{cases}=\begin{cases}
	 				\pi(\sigma_1^{-1}(y),1) & \text{if } y_{m+1}=1\\
	 				\pi(\sigma_2^{-1}(y),0) & \text{if } y_{m+1}=0
	 			\end{cases}\\
	 			&=\begin{cases}
	 				(\sigma_1(\sigma_1^{-1}(y)),1) & \text{if } y_{m+1}=1\\
	 				(\sigma_2(\sigma_2^{-1}(y)),0) & \text{if } y_{m+1}=0
	 			\end{cases}=\begin{cases}
	 				(y,1) & \text{if } y_{m+1}=1\\
	 				(y,0) & \text{if } y_{m+1}=0
	 			\end{cases}.
	 		\end{split}
	 	\end{equation}
	 	Since the latter expression is equal to $\operatorname{id}(y,y_{m+1})$, we get that $\rho=\pi^{-1}$. \qed
	 \end{proof}
 	Now, we show that as soon as a single example of permutations $\pi_i$ on $\F_2^{m}$ with the $(\mathcal{A}_m)$ property is found (here $m$ is a fixed integer), then one can always construct many such examples on $\F_2^{m+k}$, where $k>0$ is an arbitrary integer. 	
 	\begin{proposition}\label{prop: Am permutations new from old}
 		Let $\pi_j$, $\sigma_j$ for $j \in\{1,2,3\}$ be  permutations on $\F_2^m$ which satisfy the condition $\left(\mathcal{A}_m\right)$. Denote by $\pi_4=\pi_1+\pi_2+\pi_3$, $\sigma_4=\sigma_1+\sigma_2+\sigma_3$. Define four permutations $\phi_i$ for $i\in\{1,2,3,4\}$ on $\F_2^{m+1}$ as
 		\begin{equation}\label{eq: Permutations phi}
 			\phi_i(y,y_{m+1})= \begin{cases}
 				(\pi_i(y),1) & \text{if } y_{m+1}=1\\
 				(\sigma_i(y),0) & \text{if } y_{m+1}=0
 			\end{cases},\quad \text{for all } y \in \F_2^m, y_{m+1} \in \F_2,
 		\end{equation}
 		 Then, permutations $\phi_1,\phi_2,\phi_3$ satisfy the condition $\left(\mathcal{A}_m\right)$.
 	\end{proposition}
 	\begin{proof}
 		The property $(\mathcal{A}_{m+1})$ means that for the permutations $\phi_i$ on $\F_2^{m+1}$, we have that $\phi_1+\phi_2 + \phi_3=\phi_4$ is also a permutation  and $\phi_4^{-1}=\phi_1^{-1} + \phi_2^{-1} + \phi_3^{-1} $. First, we show that $\phi_4$ is a permutation. By definition of $\phi_4$, we get that for all $y\in\F_2^m$, $y_{m+1}\in\F_2$ it holds that
 		$$ \phi_4(y,y_{m+1})= \begin{cases}
 			((\pi_1+\pi_2+\pi_3)(y),1) & \text{if } y_{m+1}=1\\
 			((\sigma_1+\sigma_2+\sigma_3)(y),0) & \text{if } y_{m+1}=0
 		\end{cases}.$$
 		Since $\pi_4=\pi_1+\pi_2+\pi_3$ and $\sigma_4=\sigma_1+\sigma_2+\sigma_3$  are permutations, we get that $\phi_4$ is a permutation as well. Now, we show that $\phi_4^{-1}=\phi_1^{-1} + \phi_2^{-1} + \phi_3^{-1}$. By Lemma~\ref{eq: new permutation inverse}, we have that for all $y\in\F_2^m$, $y_{m+1}\in\F_2$ it holds that
 		\begin{equation*}
 			\begin{split}
 				\phi^{-1}_4(y,y_{m+1})&= \begin{cases}
 					((\pi_1+\pi_2+\pi_3)^{-1}(y),1) & \text{if } y_{m+1}=1\\
 					((\sigma_1+\sigma_2+\sigma_3)^{-1}(y),0) & \text{if } y_{m+1}=0
 				\end{cases} \\	
 				&= \begin{cases}
 					(\pi_1^{-1}(y)+\pi_2^{-1}(y)+\pi_3^{-1}(y),1) & \text{if } y_{m+1}=1\\
 					(\sigma_1^{-1}(y)+\sigma_2^{-1}(y)+\sigma_3^{-1}(y),0) & \text{if } y_{m+1}=0
 				\end{cases} \\	
 				&=(\phi_1^{-1} + \phi_2^{-1} + \phi_3^{-1} )(y,y_{m+1}),
 			\end{split}
 		\end{equation*} 
 		from what follows that permutations $\phi_1,\phi_2,\phi_3$ have the $\left(\mathcal{A}_{m+1}\right)$ property.\qed
 		\end{proof}
	\subsection{Concatenating Maiorana-McFarland bent functions satisfying the dual bent condition}\label{subsec: 2.1 Concatenating MM generalized}
	Theorem~\ref{th:cond_on_h} specifies the dual bent condition for Maiorana-McFarland bent functions $f_i(x,y)=Tr(x\pi_i(y))+h_i(y)$ on $\F_{2^m}\times\F_{2^m}$, where permutations $\pi_i$ have the $(\mathcal{A}_m)$ property and $f_4=f_1+f_2 +f_3$. In this section, we give a generalization of Theorem~\ref{th:cond_on_h}, which is motivated by the following example.
	
		\begin{example}\label{ex:decomposing_h}
		Define the permutations $\pi_i$ on $\F_2^4$ as follows (where the rows correspond to the coordinate functions):
		\begin{equation*}
			\begin{split}
				\pi_1(y)&=\begin{pmatrix}
					y_1 + y_2 + y_1 y_4 + y_2 y_4 + y_3 y_4 \\
					y_1 + y_1 y_2 + y_3 + y_2 y_3 + y_2 y_4 \\
					y_1 y_2 + y_3 + y_1 y_3 + y_2 y_4 + y_3 y_4 \\ 
					y_1 + y_3 + y_1 y_3 + y_2 y_3 + y_4 + y_1 y_4 + y_2 y_4
				\end{pmatrix}, \\
				\pi_2(y)&=\pi_1(y)+\begin{pmatrix} y_2 + y_3 + y_4\\ 1 + y_2 + y_3 + y_4\\ y_1 + y_3\\ y_1 + y_3 \end{pmatrix},\\
				\pi_3(y)&=\pi_1(y)+ \begin{pmatrix} y_1 + y_4\\ y_1 + y_2\\ 1 + y_1 + y_2\\ 1 + y_1 + y_4 \end{pmatrix},
				\pi_4(y)=(\pi_1+\pi_2+\pi_3)(y).
			\end{split}
		\end{equation*}
		The algebraic normal forms of the functions $h_i$ are given as follows:
		\begin{equation*}
			\begin{split}
				h_1(y)&=y_1 y_3 y_4, \quad h_2(y)= y_2 y_3 + y_1 y_4 + y_2 y_4 + y_3 y_4 + y_1 y_3 y_4, \\
				h_3(y)&=y_1 y_3 + y_2 y_3 + y_3 y_4 + y_1 y_3 y_4, \quad h_4(y)= (h_1+h_2+h_3)(y)+s(y),
			\end{split}
		\end{equation*}
		where $s(y)=y_1 + y_2 + y_4$.  
		One can check that the above defined permutations $\pi_1,\pi_2,\pi_3$ of $\F_2^4$ have the $(\mathcal{A}_4)$ property and that $f_1|| f_2|| f_3||f_4\in\mathcal{B}_{10}$ is bent for bent functions $f_{i}(x,y)=x\cdot \pi_i(y)+h_i(y)$, where $x,y\in\F_2^4$. However, the conditions of Theorem~\ref{th:cond_on_h} are not fulfilled, since $f_1+f_2+f_3+f_4=s\neq0$.
	\end{example}
	Now, we provide a more general version of Theorem~\ref{th:cond_on_h}, which covers the previous example.
	\begin{theorem}\label{th: Frobenius extended}
		Let $f_i(x, y)=Tr\left(x \pi_i(y)\right)+h_i(y)$ for $i \in\{1,2,3\}$ and $x, y \in$ $\mathbb{F}_{2^{m}}$ with $n=2m$, where the permutations $\pi_i$ satisfy the condition $\left(\mathcal{A}_m\right)$, and let $s(y)\in\mathcal{B}_{m}$. 
		Define a function $h_4\in\mathcal{B}_m$ as $h_4(y)=h_1(y)+h_2(y)+h_3(y)+s(y)$ and a bent function $f_4\in\mathcal{B}_n$ as $f_4(x,y)=f_1(x,y)+f_2(x,y)+f_3(x,y)+s(y)$. If the functions  $h_1, \ldots, h_4$  satisfy
		\begin{equation}\label{eq: hi condition}
			h_1(\pi_1^{-1}(x))+h_2(\pi_2^{-1}(x))+h_3(\pi_3^{-1}(x))+h_4((\pi_1+\pi_2+\pi_3)^{-1}(x))=1 \text {, }
		\end{equation}
		then $f_1|| f_2|| f_3||f_4\in\mathcal{B}_{n+2}$ is bent.
	\end{theorem}
	\begin{proof}
		Clearly $f_4(x,y)=(f_1+f_2+f_3+s)(x,y)=Tr\left(x\left(\pi_1+\pi_2+\pi_3\right)(y)\right)+h_4(y)$. Since the permutations $\pi_1,\pi_2,\pi_3$ satisfy the condition $\left(\mathcal{A}_m\right)$, their sum is again a permutation and $f_4$ is a bent Maiorana-McFarland function. Its dual is
		$
		f_4^*(x,y)=Tr(y\left(\pi_1+\pi_2+\pi_3\right)^{-1}(x))+h_4(\left(\pi_1+\pi_2+\pi_3\right)^{-1}(x))
		$.
		Then, the sum of dual bent functions of $f_i$ is given by
		$$
		\begin{aligned}
			(f_1^*+f_2^*+f_3^*+f_4^*)(x,y)
			=&Tr(y (\pi_1^{-1}+\pi_2^{-1}+\pi_3^{-1}+(\pi_1+\pi_2+\pi_3)^{-1})(x)) \\
			+ & h_1(\pi_1^{-1}(x))+ h_2(\pi_2^{-1}(x))  +h_3(\pi_3^{-1}(x)) \\
			+& h_4((\pi_1+\pi_2+\pi_3)^{-1}(x)) \\
			= & h_1(\pi_1^{-1}(x))+h_2(\pi_2^{-1}(x))+h_3(\pi_3^{-1}(x)) \\
			+&h_4((\pi_1+\pi_2+\pi_3)^{-1}(x))\\
			=&1,
		\end{aligned}
		$$
		thus by \cite[Theorem III.1]{HPZ2019} we have that $f_1|| f_2|| f_3||f_4\in \mathcal{B}_{n+2}$ is bent.
		\qed	
	\end{proof} 

	\begin{remark}
		Theorem~\ref{th: Frobenius extended} explains  why the concatenation $f_1|| f_2|| f_3||f_4\in\mathcal{B}_{10}$ of bent functions $f_1,f_2,f_3,f_4$ from Example~\eqref{ex:decomposing_h} is a bent function again. Now, applying Theorem~\ref{th:sharingcommonsubspace}, we show that the concatenation of these functions is outside the completed Maiorana-McFarland class. First, we observe that Maiorana-McFarland bent functions $f_i$ share the unique canonical $\mathcal{M}$-subspace $\F_2^4\times\{0_4\}$.  Consequently, all subspaces $U\subset\F_2^4\times\{0_4\}$ of dimension 3 are common $\mathcal{M}$-subspaces for all functions $f_i$; note that there are exactly ${4 \brack 3}_2=15$ such subspaces. Apart from these subspaces $U\subset\F_2^4\times\{0_4\}$, the functions $f_i$ have additionally the following 8 common $\mathcal{M}$-subspaces, which are not subspaces of $\F_2^4\times\{0_4\}$:
		$$\scalebox{1}{$\left\langle
			\begin{array}{cccccccc}
				1 & 0 & 0 & 0 & 0 & 1 & 0 & 1 \\
				0 & 1 & 1 & 0 & 0 & 0 & 0 & 0 \\
				0 & 0 & 0 & 1 & 0 & 1 & 0 & 1 \\
			\end{array}
			\right\rangle,\left\langle
			\begin{array}{cccccccc}
				1 & 0 & 0 & 0 & 1 & 1 & 1 & 1 \\
				0 & 1 & 0 & 0 & 1 & 1 & 1 & 1 \\
				0 & 0 & 1 & 1 & 0 & 0 & 0 & 0 \\
			\end{array}
			\right\rangle,\left\langle
			\begin{array}{cccccccc}
				1 & 0 & 0 & 1 & 0 & 0 & 0 & 0 \\
				0 & 1 & 0 & 0 & 0 & 1 & 0 & 1 \\
				0 & 0 & 1 & 0 & 0 & 1 & 0 & 1 \\
			\end{array}
			\right\rangle,\left\langle
			\begin{array}{cccccccc}
				1 & 0 & 0 & 1 & 0 & 0 & 0 & 0 \\
				0 & 1 & 0 & 1 & 0 & 1 & 0 & 1 \\
				0 & 0 & 1 & 1 & 0 & 1 & 0 & 1 \\
			\end{array}
			\right\rangle$}$$
		$$\scalebox{1}{$\left\langle
			\begin{array}{cccccccc}
				1 & 0 & 0 & 1 & 0 & 0 & 0 & 0 \\
				0 & 1 & 1 & 0 & 0 & 0 & 0 & 0 \\
				0 & 0 & 0 & 0 & 0 & 1 & 0 & 1 \\
			\end{array}
			\right\rangle,\left\langle
			\begin{array}{cccccccc}
				1 & 0 & 0 & 1 & 1 & 1 & 1 & 1 \\
				0 & 1 & 0 & 1 & 1 & 1 & 1 & 1 \\
				0 & 0 & 1 & 1 & 0 & 0 & 0 & 0 \\
			\end{array}
			\right\rangle,\left\langle
			\begin{array}{cccccccc}
				1 & 1 & 0 & 0 & 0 & 0 & 0 & 0 \\
				0 & 0 & 1 & 0 & 1 & 1 & 1 & 1 \\
				0 & 0 & 0 & 1 & 1 & 1 & 1 & 1 \\
			\end{array}
			\right\rangle,\left\langle
			\begin{array}{cccccccc}
				1 & 1 & 0 & 0 & 0 & 0 & 0 & 0 \\
				0 & 0 & 1 & 1 & 0 & 0 & 0 & 0 \\
				0 & 0 & 0 & 0 & 1 & 1 & 1 & 1 \\
			\end{array}
			\right\rangle$}. $$
		One can check that for every $3$-dimensional subspace  $V$ of $\F_2^8$ such that $D_aD_bf_i=0$, for all $a,b \in V$, where $i=1, \ldots ,4$, the conditions of Theorem~\ref{th:sharingcommonsubspace} are satisfied, and hence, the bent function $f_1|| f_2|| f_3||f_4\in\mathcal{B}_{10}$ is outside $\mathcal{M}^\#$.
	\end{remark}
	
	Now, we provide a recursive construction of Maiorana-McFarland bent  functions $f'_1,f'_2,f'_3,f'_4\in\mathcal{B}_{n+2}$ satisfying the condition $(f'_1)^*+(f'_2)^*+(f'_3)^*+(f'_4)^*=1$ from bent functions $f_1,f_2,f_3,f_4\in\mathcal{B}_{n}$ satisfying the condition $f_1^*+f_2^*+f_3^*+f_4^*=1$ using Theorem~\ref{th: Frobenius extended}.
	\begin{proposition}\label{prop: Extension of An}
		Let $\pi_j$, $\sigma_j$ for $j \in\{1,2,3\}$ be  permutations on $\F_2^m$ which satisfy the condition $\left(\mathcal{A}_m\right)$. Denote by $\pi_4=\pi_1+\pi_2+\pi_3$, $\sigma_4=\sigma_1+\sigma_2+\sigma_3$ and let  Boolean functions $h_i$, $g_i$ on $\F_2^m$, $i \in\{1,2,3,4\}$ satisfy
		\begin{equation*}
			\begin{split}
				h_1(\pi_1^{-1}(y))+h_2(\pi_2^{-1}(y))+h_3(\pi_3^{-1}(y))+h_4(\pi_4^{-1}(y))=1, \\
				g_1(\sigma_1^{-1}(y))+g_2(\sigma_2^{-1}(y))+g_3(\sigma_3^{-1}(y))+g_4(\sigma_4^{-1}(y))=1. 
			\end{split}
		\end{equation*}
		Define four permutations $\phi_i$ on $\F_2^{m+1}$ as in~\eqref{eq: Permutations phi}
		and four Boolean functions $h_i'$ on $\F_2^{m+1}$ as follows
		\begin{equation*}
			h_i'(y,y_{m+1})=y_{m+1}h_i(y)+(y_{m+1}+1)g_i(y) \mbox{ for } i\in\{1,2,3,4\}.
		\end{equation*} 
		Then, the following hold.
		\begin{itemize}
			\item[1.]  The functions $h'_i$ satisfy
			\begin{equation*}
				\sum\limits_{i=1}^4 	h_i'(\phi_i^{-1}(y,y_{m+1}))=1,
			\end{equation*}
			for all $y\in\F_2^m,y_{m+1}\in\F_2$, where $\phi_4=\phi_1+\phi_2+\phi_3$.
			\item[2.] The Boolean functions 	$f'_i(x',y')=Tr\left(x' \phi_i(y')\right)+h'_i(y')$ for $i \in\{1,2,3,4\}$ and $x'=(x,x_{m+1}), y'=(y,y_{m+1}) \in$ $\mathbb{F}_{2}^{m+1}$ are bent, moreover,  $f_1'|| f_2'|| f_3'||f_4'\in\mathcal{B}_{n+2}$ is bent as well.
		\end{itemize}
	\end{proposition}
	\begin{proof}
		\textit{1.} Observe that for $j\in\{1,2,3\}$, we have that for all $y\in\F_2^m,y_{m+1}\in\F_2$ holds
		\begin{equation*}
			\begin{split}
				h'_j(\phi_j^{-1}(y,y_{m+1}))&= 
				\begin{cases}
					h'_j(\phi^{-1}_j(y,1)) & \text{if } y_{m+1}=1\\
					h'_j(\phi^{-1}_j(y,0)) & \text{if } y_{m+1}=0
				\end{cases}\\
				&=\begin{cases}
					h'_j(\pi_j^{-1}(y),1) & \text{if } y_{m+1}=1\\
					h'_j(\sigma_j^{-1}(y),0) & \text{if } y_{m+1}=0
				\end{cases} \\
				&=\begin{cases}
					h_j(\pi_j^{-1}(y)) & \text{if } y_{m+1}=1\\
					g_j(\sigma_j^{-1}(y)) & \text{if } y_{m+1}=0
				\end{cases}. \\
			\end{split}
		\end{equation*}
		\noindent	Then, for all $y\in\F_2^m,y_{m+1}\in\F_2$, we consider the sum
		\begin{equation*}
			\begin{split}
				\sum\limits_{i=1}^4 	h_i'(\phi_i^{-1}(y,y_{m+1}))=&\begin{cases}
					\sum\limits_{i=1}^4 	h_i(\pi_i^{-1}(y)) & \text{if } y_{m+1}=1\\
					\sum\limits_{i=1}^4 g_i(\sigma_i^{-1}(y)) & \text{if } y_{m+1}=0
				\end{cases}\\
				=&\begin{cases}
					\sum\limits_{i=1}^3 	h_i(\pi_i^{-1}(y))+h_4((\pi_1+\pi_2+\pi_3)^{-1}(y)) & \text{if } y_{m+1}=1\\
					\sum\limits_{i=1}^3 g_i(\sigma_i^{-1}(y))+ g_4((\sigma_1+\sigma_2+\sigma_3)^{-1}(y)) & \text{if } y_{m+1}=0
				\end{cases}\\
				=&1,
			\end{split}
		\end{equation*}
		since $h_1\left(\pi_1^{-1}(y)\right)+h_2\left(\pi_2^{-1}(y)\right)+h_3\left(\pi_3^{-1}(y)\right)+h_4((\pi_1+\pi_2+\pi_3)^{-1}(y))=1$ and $g_1\left(\sigma_1^{-1}(y)\right)+g_2\left(\sigma_2^{-1}(y)\right)+g_3\left(\sigma_3^{-1}(y)\right)+g_4((\sigma_1+\sigma_2+\sigma_3)^{-1}(y))=1$ hold for all $y\in\F_2^m$.
		
		\noindent\textit{2.} The statement follows immediately from Theorem~\ref{th: Frobenius extended}. \qed
	\end{proof}
	\begin{remark}
		Applying Proposition~\ref{prop: Extension of An} to Example~\ref{ex:decomposing_h}, one gets infinite families of permutations having the $(\mathcal{A}_{m+k})$ property for $k\in\N$, since Proposition~\ref{prop: Extension of An} can be used recursively to define permutations on $\F_2^{m+k}$ that satisfy the $(\mathcal{A}_{m+k})$ property. 	
	\end{remark}
\begin{openproblem}
	Provide sufficient conditions for  permutations $\phi_i$ in Proposition \ref{prop: Extension of An}, more precisely  on the permutations $\pi_i$ and $\sigma_i$ defined on $\F_2^m$, so that the concatenation of four bent functions $f_i(x,y)=x \cdot \phi_i(y) + h_i(y) \in \cM^\#$, where $f_i:\F_2^{m+k} \times  \F_2^{m+k} \rightarrow \F_2$, generates a bent function $f=f_1||f_2||f_3||f_4$ on $\F_2^{2m+2k +2}$ outside $\cM^\#$.
	\end{openproblem}
\begin{openproblem}
	Example~\ref{ex:decomposing_h} demonstrates that it is possible to define four permutations $\pi_i$ of $\F_2^m$ with $i\in\{1,2,3,4\}$, where $\pi_j(y)=\pi_1(y)+L_j(y)$ for $j=2,3,4$, and $L_j\colon\F_2^m\to\F_2^m$ are linear mappings, such that for four suitably chosen Boolean functions $h_i\in\mathcal{B}_m$, the concatenation $f=f_1||f_2||f_3||f_4\in\mathcal{B}_{2m+2}$ of four Maiorana-McFarland bent functions $f_i(x,y)=x\cdot\pi_i(y)+h_i(y)$ on $\F_2^m\times\F_2^m$ is bent. We suggest to find constructions of such permutations $\pi_i$ and Boolean functions $h_i$ explicitly, and, additionally, show that the resulting bent function $f\in\mathcal{B}_{2m+2}$ is outside $\mathcal{M}^\#$. 
\end{openproblem}
In the sequel, we will demonstrate that these sufficient conditions that the concatenation $f=f_1||f_2||f_3||f_4$ is outside $\cM^\#$ can be specified when permutations  $\pi_i$ on $\F_2^m$ have the  $(\mathcal{A}_{m})$ property, thus without using Proposition  \ref{prop: Extension of An} recursively.


			\section{A sufficient condition for satisfying the outside $\cM^\#$ property}\label{sec: 3}
		
			
			Since in the bent 4-concatenation we consider bent functions $f_i\in\mathcal{B}_n$ in $\mathcal{M}^\#$, it is essential to specify the conditions on these functions such that the resulting bent function $f=f_1||f_2||f_3||f_4\in\mathcal{B}_{n+2}$ is outside $\mathcal{M}^\#$. Otherwise, one just gets a complicated construction method of bent functions in $\mathcal{M}^\#$. For this purpose, we will use the following description of $\mathcal{M}$-subspaces of $f=f_1||f_2||f_3||f_4\in\mathcal{B}_{n+2}$, which together with several other results for the 4-bent concatenation of bent functions in $\cM^\#$ class are available in~\cite{PPKZ2023}.
			
			\begin{proposition}\cite{PPKZ2023} \label{prop:commonsubspace} Let $f_1,f_2,f_3,f_4\in\mathcal{B}_n$ be four Boolean functions (not necessarily bent), such that $f=f_1||f_2||f_3||f_4 \in\B_{n+2}$ is a bent function in $\cM^{\#}$. Let $W\subset\F_2^{n+2}$ be an $\mathcal{M}$-subspace of $f$ of dimension $(\frac{n}{2}+1)$. Then, there exists an $(\frac{n}{2}-1)$-dimensional subspace $V$ of $\F_2^n$ such that $V \times \{(0,0)\}$ is a subspace of $W$, and such that for all $i=1, \ldots ,4$ the equality $D_aD_bf_i=0$ holds for all $a,b \in V$.
			\end{proposition}
			
			For the main result of this section, we will also need to define the~$(P_1)$ property, which was recently introduced in~\cite{PPKZ2023} for specifying Maiorana-McFarland bent functions with the unique canonical $\mathcal{M}$-subspace.
		\begin{definition}	 A mapping $\pi\colon\F_2^m\to \F_2^m$ has the property $(P_1)$ if for all linearly independent  $v,w\in\F_2^m$ we have that $D_vD_w\pi(y)\neq0_m$ for all $y\in\F_2^m$. 
			\end{definition}
		More precisely, by Theorem 3.1 in \cite{PPKZ2023}, if a permutation $\pi$ on $\F_2^m$ satisfies the property $(P_1)$ then the bent function 
		$f(x,y)=x \cdot \pi(y) + h(y)$ (where $h$ is arbitrary) has the unique canonical $\mathcal{M}$-subspace of dimension $m$, namely $U=\F_2^m \times \{0_m\}$. 
			
			\begin{theorem}\label{th: Outside MM}
                Let  $n=2m$ for $m >3$ and define three bent functions $f_i(x,y)=x \cdot \pi_i(y)  + h_i(y)$, with $x, y \in \F_2^m$, for  $i=1, \ldots,3$, where
                \begin{itemize}
                	\item[1.] Permutations $\pi_1,\pi_2,\pi_3$ have the property $\left(\mathcal{A}_m\right)$,
                	\item[2.] Permutations $\pi_1,\pi_2,\pi_3$ and the mapping $\pi_1+\pi_2$ have the property~$(P_1)$,
                	\item[3.] The components of $\pi_1+\pi_2$ do not admit linear structures.
                \end{itemize} 
				Define $f=f_1||f_2||f_3||f_4$ where   $f_4(x,y)=f_1(x,y)+f_2(x,y)+f_3(x,y)   + s(y)$ (consequently $h_4(y)=h_1(y)+h_2(y)+h_3(y) + s(y)$) using suitable $h_i\in\mathcal{B}_m$ and $s\in\mathcal{B}_m$ so that 
				the dual bent condition in~\eqref{eq: hi condition} is satisfied.
				Then, the  function $f\in\mathcal{B}_{n+2}$ is bent and  outside $\cM^\#$. In particular, the same conclusion is valid  when $s(y)=0$.
			\end{theorem}
			\begin{proof}  The bentness of $f$ follows from Theorem \ref{th: Frobenius extended}, since we assume that the dual bent condition in  \eqref{eq: hi condition} is satisfied. To simplify the notation, we use the variable $z \in \F_2^n$ to replace $(x,y) \in \F_2^m \times \F_2^m$. 
				Denoting   $a=(a',a^{(1)},a^{(2)})$ and $b=(b',b^{(1)},b^{(2)})$,  where  $a',b' \in \F_2^n$   and $a^{(i)},b^{(i)} \in \F_2$, the second-order derivative of $f$ is given by $D_aD_bf(z,y_1,y_2)=$
				\begin{equation}\label{eq:2ndderiv_conc correct}
					\begin{split}
						&=D_{a'}D_{b'}f_1(z)+ y_1D_{a'}D_{b'}f_{13}(z)+ y_2 D_{a'}D_{b'}f_{12}(z)+  y_1y_2D_{a'}D_{b'}f_{1234}(z) \\
						& + a^{(1)}D_{b'}f_{13}(z +  a')+  b^{(1)} D_{a'}f_{13}(z +  b') +  a^{(2)}D_{b'}f_{12}(yz +  a')   \\
						& +  b^{(2)}D_{a'}f_{12}(z +  b') +  (a^{(1)}y_2 +  a^{(2)}y_1 +  a^{(1)}a^{(2)})D_{b'}f_{1234}(z +  a')   \\ & +(b^{(1)}y_2 +  b^{(2)}y_1 +  b^{(1)}b^{(2)})\times D_{a'}f_{1234}(z +  b')  \\ & +  (a^{(1)}b^{(2)} +  b^{(1)}a^{(2)})f_{1234}  (z +  a' +  b'), 
					\end{split}
				\end{equation}
				where $f_{i_1\ldots i_k}:=f_{i_1} +  \cdots  +  f_{i_k}$.	
					Since $D_uD_v  \pi_i(y) \neq 0_m$ for any nonzero $u \neq v \in \F_2^m$ (as $\pi_i$ satisfies the property~$(P_1)$), the functions $f_i$ share the unique canonical $\mathcal{M}$-subspace $U=\F_2^m \times \{0_m\}$, see  Theorem 3.1 in  \cite{PPKZ2023}.
					For convenience, we denote $a'=(a_1,a_2)$ and  $b'=(b_1,b_2)$, where $a_i,b_i \in \F_2^m$. Since, by the assumption,  $D_{a_2}D_{b_2} ( \pi_1(y) + \pi_2(y)) \neq 0_m$ for any $a_2,b_2 \in \F_2^m$ ($a_2,b_2\ne0$ and distinct),
                      the term $y_2D_{a'}D_{b'}f_{12}(x,y)$
					in \eqref{eq:2ndderiv_conc correct}  cannot be canceled unless $a_2=0_m$ or $b_2=0_m$ or $a_2=b_2$, which  is due to the fact that (same can be deduced for $D_{(a_1, a_2)}D_{(b_1, b_2)}f_{13}(x,y)$)
					\begin{equation}\label{eq:2ndorderf12}
						\begin{split} 
							D_{(a_1, a_2)}D_{(b_1, b_2)}f_{12}(x,y)
							=&x\cdot\left( D_{a_2}D_{b_2}(\pi_1(y)+\pi_2(y)) \right)\\
							+&  a_1\cdot D_{b_2}(\pi_1+\pi_2)(y+ a_2)\\
							+& b_1 \cdot D_{a_2}(\pi_1+\pi_2)(y+ b_2) + D_{a_2} D_{b_2}  h_{12}(y).
						\end{split} 
					\end{equation}
				Notice that the other terms in Equation~\eqref{eq:2ndderiv_conc correct} that possibly contain the variable $y_2$, such as  $a^{(1)}y_2 D_{b'}f_{1234}(z +  a')$ and   $b^{(1)}y_2 D_{a'}f_{1234}(z +  b')$, cannot cancel the term  $y_2D_{a'}D_{b'}f_{12}(x,y)$ since $f_{1234}(x,y)=s(y)$ and its derivatives never depend on $x$.
					Thus, for any $a=(a_1,a_2,a^{(1)},a^{(2)})$ and $b=(b_1,b_2,b^{(1)},b^{(2)})$ in some $(m+1)$-dimensional subspace $W$ of $\F_2^{2m+2}$, we necessarily have that either $a_2=0_m$ or $b_2=0_m$, alternatively $a_2=b_2$. 
					
					Since the functions $f_i$ share the unique canonical $\mathcal{M}$-subspace $U=\F_2^m \times \{0_m\}$, any other subspace $V$ of $\F_2^m \times \F_2^m$ for which $D_{a'}D_{b'}f_i(x,y)=0$ for all $a',b' \in V$ must have dimension less than $m$. By Proposition~\ref{prop:commonsubspace}, if $f$ defined on $\F_2^{2m+2}$ belongs to $\cM^\#$ then for any $\mathcal{M}$-subspace $W$ of $f$ of dimension $m+1$ there must exist $V \subset \F_2^{2m}$ of dimension 
					$m-1$ such  that $D_aD_b f_i =0$ for all $i=1,\ldots,4$ and any $a,b \in V$. Furthermore, by Proposition~\ref{prop:commonsubspace}, we have that $V \times (0,0)$ is a subspace of $W$. There are only two possibilities for $V$, i.e., either $V \subset U = \F_2^m \times \{0_m\}$ or $V \not \subset U$.  
					
					We first consider the case that $V \subset U = \F_2^m \times \{0_m\}$, where $\dim(V)=m-1$. Then, $V \times (0,0) \subset W$ and to extend this subspace to $W$,  we need to adjoin two elements of $\F_2^{2m+2}$, say $u=(u_1,u_2,u^{(1)},u^{(2)}),v=(v_1,v_2,v^{(1)},v^{(2)}) \in \F_2^m \times \F_2^m \times \F_2 \times \F_2$,  and  we further denote  $u'=(u_1,u_2)$, $v'=(v_1,v_2)$. Then, we cannot have the case that $u_2=v_2=0_m$ since this would imply that $f_{12}$ on $\F_2^n$ has an $\mathcal{M}$-subspace of dimension $n/2+1$ which is impossible 
					(see for instance \cite[Proposition 3.1] {Homogeneous2022}).
					 On the other hand, if $u_2 \neq v_2 \neq 0_m$ then again $y_2D_{u'}D_{v'}f_{12}(x,y)$ cannot be canceled in \eqref{eq:2ndderiv_conc correct} due to  the term $y_2 x\cdot\left( D_{u_2}D_{v_2}(\pi_1(y)+\pi_2(y)) \right) $. 
					
					W.l.o.g. we now  assume that 
					$u_2=0_m$ and $v_2 \neq 0_m$, which implies that $U \times (0,0) \subset W$. Hence, $W=\langle U \times (0,0), v \rangle $, where $v_2 \neq 0_m$ and $v=(v_1,v_2, v^{(1)},v^{(2)} )$ with $v^{(1)},v^{(2)} \in \F_2$. Notice that the case $u_2=v_2$, which also might lead to $D_{u'}D_{v'}f_{12}(x,y)=0$, reduces to this case since $u_2+v_2=0_m$ and then $u'  + v' \in U$. 
					Now, we note that in $W=\langle U \times (0,0), v \rangle $ there must exist an element $z=(z',0,0)$
					 (where $z'=(z_1,z_2) \in \F_2^m \times \F_2^m$) such that $z_1=v_1$ and consequently $z' + 
					v'=(0_m,v_2)$. Considering \eqref{eq:2ndorderf12}, and replacing $a' \rightarrow z'=(v_1, 0_m)$ and  $b' \rightarrow (0_m,v_2)$, we have that only the term $v_1\cdot D_{v_2}(\pi_1+\pi_2)(y)$ remains in  \eqref{eq:2ndorderf12}, which cannot be zero due to our assumption that the  components of $(\pi_1+\pi_2)(y)$ do not admit linear structures.
					
					The second case arises when $V \not \subset U$, where $\dim(V)=m-1$. Hence, $V$ contains at least one element $ a' =(a_1,a_2) \not \in U$, so that $a_2 \neq 0_m$. If $V$ contains one more element not in  $U$, say  $ b'$, then $D_{a'}D_{b'}f_{12}(x,y)\neq 0$ (again the term $x\cdot\left( D_{a_2}D_{b_2}(\pi_1(y)+\pi_2(y)) \right)$ in  \eqref{eq:2ndorderf12} survives since $a_2,b_2 \neq 0_m$) and consequently $D_aD_bf(x,y,y_1,y_2)\neq 0$. 
					If $V$ does not contain one more element which is not in  $U$, then it can be extended  to  $U$ (by replacing $a'$ with  some $(u_1,0_m)$) and the above arguments apply.\qed
				\end{proof}
			Notice that we consider permutations $\pi$  on $\F_2^m$ for $m>3$ since otherwise $\pi$ are at most quadratic and, hence, their components admit linear structures~\cite{CharKyure}.
				\begin{example}\label{ex: bent from monomials outside MM}
					Let $m = 4$ and the multiplicative group of $\F_{2^{4}}$ be given by $\F_{2^{4}}^*=\langle a \rangle$, where
					the primitive element $a$ satisfies $a^4 + a + 1=0$. Let $d=14$, which satisfies $d^2 \equiv 1 \mod 15$. Define $\alpha_1=a, \alpha_2=a^2,\alpha_3=a^4$ and $\alpha_4=\alpha_1+\alpha_2+\alpha_3=a^8$. It is easy  to verify  that for  $i=1, \ldots,3$, the  permutations $\pi_i$ satisfy the three conditions in Theorem \ref{th: Outside MM}.
					Define the following four Boolean functions $$h_1(y)=0, \; h_2(y)=Tr(y),\; h_3(y)=Tr(a y), \; h_4(y)=Tr(a^{13}y) + 1,$$ as well as four bent Maiorana-McFarland bent functions $$f_i(x,y)=Tr(x\pi_i(y))+h_i(y)\quad \mbox{for } i=1,2,3,4\quad \mbox{where } x,y\in\F_{2^4}.$$  Note that $h_1(y)+h_2(y)+h_3(y)+h_4(y)= s(y)= Tr(a^{11}y) + 1$, and hence, $f_4=f_1+f_2+f_3+ s $. Since the functions $h_i$  satisfy the condition~\eqref{eq: hi condition} of Theorem~\ref{th: Frobenius extended}, we have that $f=f_1|| f_2|| f_3||f_4\in\mathcal{B}_{10}$. By Theorem~\ref{th: Outside MM}, the function $f$ is bent and outside $\mathcal{M}^\#$.
				\end{example}

				\section{Explicit construction methods of bent functions outside $\mathcal{M}^\#$ using permutation monomials}\label{sec: 4}
				Monomial permutations $\pi_i$ of $\F_{2^m}$ satisfying the $(\mathcal{A}_m)$ property were specified in~\cite{Mesnager2015}. In this section, we show that for these permutations, it is always possible to find suitable functions $h_i$, such that the concatenation of bent functions $f_i(x,y)=Tr(x\pi_i(y))+h_i(y)$ is bent and outside $\mathcal{M}^\#$.
				\begin{theorem}\cite{Mesnager2015} \label{th:permutA_m2nd class}
					Let $m \geq 3$ be an integer and $d^2 \equiv 1 \mod{2^m-1}$. Let $\pi_i$ be three permutations of $\F_{2^m}$ defined by $\pi_i(y)=\alpha_i y^d$, for $i=1,2,3$, where $\alpha_i \in \F_{2^m}^*$ are pairwise distinct  elements such that $\alpha_i^{d+1}=1$ and $\alpha_4^{d+1}=1$ where $\alpha_4=\alpha_1+\alpha_2 + \alpha_3$. Then, the permutations $\pi_i$ of $\F_{2^m}$ satisfy the property ($\mathcal{A}_m$) and furthermore $\pi_i$ are involutions (i.e., $\pi_i^{-1}=\pi_i$)
					 as well as $\pi_4=\pi_1+\pi_2+\pi_3$.
				\end{theorem}
				\subsection{Concatenating Maiorana-McFarland bent functions stemming from the permutation monomials}
				In the following statement, we show how one can specify the functions $h_i$ for permutations $\pi_i$ in Theorem~\ref{th:permutA_m2nd class}, so that the dual bent condition for the corresponding Maiorana-McFarland bent functions $f_i(x,y)=T(x\pi_i(y))+h_i(y)$ is fulfilled.
				\begin{proposition}\label{th: concatenating monomials}
					Let $m \ge 3$ and $\pi_i(y)=\alpha_iy^d$ for $i=1,2,3,4$ be involutions of $\F_{2^m}$ defined as in Theorem~\ref{th:permutA_m2nd class}. Let also $\sigma\colon\{\alpha_1,\alpha_2,\alpha_3,\alpha_4\}\to \{\alpha_1,\alpha_2,\alpha_3,\alpha_4\}$ be a permutation. Define Boolean functions $h_i\in\mathcal{B}_{m}$ for $i=1,2,3,4$  as follows
					\begin{equation}
						h_i(y)=Tr(\beta_i y^k)\quad\mbox{for }i=1,2,3\quad\mbox{and } h_4(y)=Tr(\beta_4 y^k) + 1,
					\end{equation}
					where $k\in\N$ and the elements $\beta_i\in\F_{2^m}^*$ are given by 
					\begin{equation}
						\beta_1=\frac{\sigma(\alpha_2)}{\alpha_1^k}, \beta_2=\frac{\sigma(\alpha_3)}{\alpha_2^k}, \beta_3=\frac{\sigma(\alpha_4)}{\alpha_3^k}, \beta_4=\frac{\sigma(\alpha_1)}{\alpha_4^k}.
					\end{equation}
					Then, the Maiorana-McFarland bent functions $f_i(x, y)=Tr\left(x \pi_i(y)\right)+h_i(y)$ for $i \in\{1,2,3,4\}$ and $x, y \in$ $\mathbb{F}_{2^{m}}$ satisfy
					\begin{equation*}
						h_1(\pi_1^{-1}(y))+h_2(\pi_2^{-1}(y))+h_3(\pi_3^{-1}(y))+h_4(\pi_4^{-1}(y))=1,
					\end{equation*}
					and, hence, $f=f_1|| f_2|| f_3||f_4\in\mathcal{B}_{2m+2}$ is bent. 
				\end{proposition}
				\begin{proof}
				Since all permutations $\pi_i(y)=\alpha_i y^d$ of $\F_{2^m}$ are involutions, we have 
					\begin{equation*}
						\begin{split}
							\sum\limits_{i=1}^4 	h_i(\pi_i^{-1}(y))&=Tr(\beta_1\alpha_1^ky^{kd}+\beta_2\alpha_2^ky^{kd}+\beta_3\alpha_3^ky^{kd}+\beta_4\alpha_4^ky^{kd})+1\\
							&=Tr((\sigma(\alpha_1)+\sigma(\alpha_2)+\sigma(\alpha_3)+\sigma(\alpha_4))y^{kd})+1=1,
						\end{split}
					\end{equation*}
				since  $\sigma(\alpha_1)+\sigma(\alpha_2)+\sigma(\alpha_3)+\sigma(\alpha_4)=\alpha_1+\alpha_2+\alpha_3+\alpha_4=0$. Consequently, $f=f_1|| f_2|| f_3||f_4\in\mathcal{B}_{2m+2}$ is bent.\qed
				\end{proof}
			\subsection{Bent functions outside $\cM^\#$ using monomial permutations}
			 In this subsection, we describe how one can concatenate Maiorana-McFarland bent functions stemming from non-linear permutation monomials in order to get infinite families of bent functions outside $\cM^\#$. While the case of non-quadratic monomials is an immediate consequence of Theorem~\ref{th: Outside MM}, the quadratic case cannot be covered by this result anymore. The reason for that is the condition ``the components of $\pi$ do not admit linear structures'' excludes the use of quadratic monomial permutations~\cite{CharKyure}. Therefore, the quadratic case needs a special treatment. In this context, we recall  that a quadratic permutation $\pi$ on $\F_2^m$  satisfies the $(P_1)$ property if and only if $\pi$  is an APN permutation, see \cite[Corollary 4.2]{PPKZ2023}. Fist, we recall the following characterization of linear structures of the components of permutation  monomials given in~\cite{CharKyure} (stated only for binary quadratic case), which is useful for our purpose.
			 \begin{theorem}  \cite[Theorem 5]{CharKyure} \label{th:Pas_Goh}
			 	Let $\delta \in \F_{2^m}$ and $1 \leq s \leq 2^m-2$ be such that $f(x)=Tr(\delta x^s)$ is not the zero function  on $\F_2^m$. Then, when $wt_H(s)=2$ the function $f$ has a linear structure if and only if the following is true:\\
			 	(ii): $s=2^j(2^i+1)$, 
			 	where $0 \leq i, j \leq m-1$, 
			 	$i \not \in \{0,m/2\}$.
			 	In this case,  $\alpha \in \F_{2^m}$ is a linear structure of $f$ if and only if it satisfies $(\delta^{2^{m-j}} \alpha^{2^i+1})^{2^i-1}+1=0$. More exactly the linear space $\Lambda$ of $f$ is as follows. 
			 	Denote $\sigma =\gcd(m,2i)$. Then, $\Lambda=\{0\}$ if $\delta$ is not a $(2^i+1)$-th power in $\F_{2^m}$. Otherwise, if $\delta=\beta^{2^j(2^i+1)}$ for some $\beta \in \F_{2^m}$, it holds that $\Lambda=\beta^{-1}\F_{2^\sigma}$.
			 \end{theorem}
	With this result, which states that and only linear and quadratic monomial permutations always have some components that admit 
	non-zero linear structures, we are ready to prove the main theorem of this subsection.		 	
	\begin{theorem}\label{th:quadperm}
					Let $m \geq 3$ and the Maiorana-McFarland bent functions $$f_i(x, y)=Tr\left(x \pi_i(y)\right)+h_i(y)$$ for $i \in\{1,2,3,4\}$ and $x, y \in$ $\mathbb{F}_{2^{m}}$ be defined as in Proposition~\ref{th: concatenating monomials}, where $\pi_i(y)=\alpha_iy^d$ are four  permutations of $\F_{2^m}$ defined as in Theorem~\ref{th:permutA_m2nd class} satisfying  additionally the $(P_1)$ property. If $wt(d)>1$, then $f=f_1|| f_2|| f_3||f_4\in\mathcal{B}_{2m+2}$ is bent and outside $\cM^\#$.
				\end{theorem}	
				\begin{proof}
				For  $wt(d)>2$,	the statement follows from Theorem~\ref{th: Outside MM} and  Theorem~\ref{th:Pas_Goh}. 
					
					For  $wt(d)=2$, without loss of generality, we assume $d=2^j(2^i+1)$.	Since all $\pi_i$ have the property~$(P_1)$ (and are necessarily APN permutations by Corollary 4.2 in \cite{PPKZ2023}), we have that by Theorem 3.1 in~\cite{PPKZ2023}, the functions $f_i$ share the unique canonical $\mathcal{M}$-subspace $U=\F_2^m \times \{0_m\}$ of dimension $m$. Thus,  we can apply Theorem   \ref{th:sharingcommonsubspace} and verify that the conditions $1)$ to $3)$ are satisfied.
				
						Let $V$ be an $(\frac{n}{2}-1)$-dimensional subspace  of $ \F_2^{2m}$  such that $D_aD_bf_i=0$, for all $a,b \in V$; $i=1, \ldots ,4.$
					If for any $v\in  \F_2^{2m} $ and  any such $V \subset  \F_2^{2m}$, there exist $u^{(1)},u^{(2)},u^{(3)}\in V $ such that the following three conditions hold simultaneously (replacing $x \in \F_2^{2m}$ in Theorem   \ref{th:sharingcommonsubspace} by $(x,y) \in \F_2^m  \times \F_2^m$): 
					\begin{enumerate}
						\item[1)] $D_{u^{(1)}}f_1(x,y)+D_{u^{(1)}}f_2((x,y)+v)\neq 0,~\textit{or}~D_{u^{(1)}}f_3(x,y)+D_{u^{(1)}}f_4((x,y)+v)\neq 0,$
						\item[2)] $D_{u^{(2)}}f_1(x,y)+D_{u^{(2)}}f_3((x,y)+v)\neq 0,~\textit{or}~D_{u^{(2)}}f_2(x,y)+D_{u^{(2)}}f_4((x,y)+v)\neq 0,$
						\item[3)] $D_{u^{(3)}}f_2(x,y)+D_{u^{(3)}}f_3((x,y)+v)\neq 0,~\textit{or}~D_{u^{(3)}}f_1(x,y)+D_{u^{(3)}}f_4((x,y)+v)\neq 0,$ \\
				then by  Theorem \ref{th:sharingcommonsubspace} we conclude that $f$ is outside $\cM^\#$. 
				\end{enumerate}
		
	In what follows, we show  that the above conditions are actually satisfied. 	Denoting $u^{(i)}=(u_1^{(i)},u_2^{(i)}) \neq (0_m,0_m)$,  the translates of  $f_i$ correspond to,
		\begin{equation}\label{equ quadratic1 corollary}
			\begin{array}{rl}
				&	f_i(x+u^{(1)}_1,y+u^{(1)}_2)\\=&Tr\left((x+u^{(1)}_1) \pi_i(y+u^{(1)}_2)\right)+h_i(y+u^{(1)}_2)\\
				=&Tr\left((x+u^{(1)}_1) \alpha_i(y+u^{(1)}_2)^d\right)+h_i(y+u^{(1)}_2)\\
				=&Tr\left((x+u^{(1)}_1) \alpha_i\left((y)^d+(u^{(1)}_2)^d+(y)^{2^{i+j}}(u^{(1)}_2)^{2^j}+(y)^{2^{j}}(u^{(1)}_2)^{2^{i+j}}\right)\right)\\
				+&h_i(y+u^{(1)}_2)\\
				=&Tr\left(x \alpha_i\left((y)^d+(u^{(1)}_2)^d+(y)^{2^{i+j}}(u^{(1)}_2)^{2^j}+(y)^{2^{j}}(u^{(1)}_2)^{2^{i+j}}\right)\right)+G(y), 
				
			\end{array}
		\end{equation}
		where 
		\begin{equation*}
			\begin{split}
				G(y)&= Tr\left(u^{(1)}_1 \alpha_i\left((y)^d+(u^{(1)}_2)^d+(y)^{2^{i+j}}(u^{(1)}_2)^{2^j}+(y)^{2^{j}}(u^{(1)}_2)^{2^{i+j}}\right)\right)\\ & +h_i(y+u^{(1)}_2).
			\end{split}
		\end{equation*}
		Then,
			\begin{equation}\label{equ quadratic2 corollary}
				\begin{array}{rl}
				D_{u^{(1)}}f_i(x,y)=&Tr\left(x \pi_i(y)\right)+h_i(y)+Tr\left((x+u^{(1)}_1) \pi_i(y+u^{(1)}_2)\right)+h_i(y+u^{(1)}_2)\\
				=&Tr\left(x \alpha_i\left((u^{(1)}_2)^d+(y)^{2^{i+j}}(u^{(1)}_2)^{2^j}+(y)^{2^{j}}(u^{(1)}_2)^{2^{i+j}}\right)\right)\\
				+&G(y)+h_i(y),
				\end{array}
			\end{equation}
		and  similarly
		\begin{equation}\label{equ quadratic3 corollary}
			\begin{array}{rl}
			&	D_{u^{(1)}}f_i(x+v_1,y+v_2)			\\	=&Tr\left((x+v_1) \alpha_i\left((u^{(1)}_2)^d+(y+v_2)^{2^{i+j}}(u^{(1)}_2)^{2^j}+(y+v_2)^{2^{j}}(u^{(1)}_2)^{2^{i+j}}\right)\right)\\
			+&G(y+v_2)+h_i(y+v_2).		
				
			\end{array}
		\end{equation}
		From  (\ref{equ quadratic2 corollary}) and (\ref{equ quadratic3 corollary}), we have  
			\begin{equation}\label{equ quadratic4 corollary}
			\begin{array}{rl}
				&	D_{u^{(1)}}f_1(x,y)	+D_{u^{(1)}}f_2(x+v_1,y+v_2)			\\		=&Tr\left(x \alpha_1\left((u^{(1)}_2)^d+(y)^{2^{i+j}}(u^{(1)}_2)^{2^j}+(y)^{2^{j}}(u^{(1)}_2)^{2^{i+j}}\right)\right)+G(y)+h_i(y)\\+&Tr\left((x+v_1) \alpha_2\left((u^{(1)}_2)^d+(y+v_2)^{2^{i+j}}(u^{(1)}_2)^{2^j}+(y+v_2)^{2^{j}}(u^{(1)}_2)^{2^{i+j}}\right)\right)\\+&G(y+v_2)+h_i(y+v_2),\\
			\end{array}
		\end{equation}
	and furthermore (\ref{equ quadratic4 corollary}) can be written as 
		$$Tr\left(x (\alpha_1+\alpha_2)\left((y)^{2^{i+j}}(u^{(1)}_2)^{2^j}+(y)^{2^{j}}(u^{(1)}_2)^{2^{i+j}}\right)\right)+G'(y)+S(x),$$
		where $G'(y)$ and $S(x)$
		 cover   the terms in the ANF of $D_{u^{(1)}}f_1(x,y)	+D_{u^{(1)}}f_2(x+v_1,y+v_2)	 $ that exclusively  use the variables $y$ and $x$, respectively.  Thus, 
		we have 
		$$	D_{u^{(1)}}f_1(x,y)	+D_{u^{(1)}}f_2(x+v_1,y+v_2)	\neq 0$$
		since   if $\alpha_1\neq \alpha_2 $, then $$Tr\left(x (\alpha_1+\alpha_2)\left((y)^{2^{i+j}}(u^{(1)}_2)^{2^j}+(y)^{2^{j}}(u^{(1)}_2)^{2^{i+j}}\right)\right)\neq const.,$$
		where the functions $G'(y)$ and $S(x)$ clearly cannot cancel this term since it contains quadratic terms of the form $Tr(xy^{2^{i+j}})	$ or 
		$Tr(xy^{2^{j}})	$.
		Similarly, we have 
			$$	D_{u^{(2)}}f_1(x,y)	+D_{u^{(2)}}f_3(x+v_1,y+v_2)	\neq 0$$
		if  $\alpha_1\neq \alpha_3 $,	and 
				$$	D_{u^{(3)}}f_2(x,y)	+D_{u^{(3)}}f_3(x+v_1,y+v_2)	\neq 0$$
				 if $\alpha_2\neq \alpha_3 $.
				
				Thus, by Theorem \ref{th:sharingcommonsubspace},
					 $f$ is outside $\cM^\#$. \qed
				\end{proof}
			\begin{remark}
				\emph{1.} Notice that there are no quadratic APN permutations  on $\F_2^m$ for even $m$ \cite{Seberry94}, but we do  not refine the parity of $m$ in Theorem \ref{th:quadperm} since the statement is given in a wider context.   \ \\
				\noindent\emph{2.} Using Theorem~\ref{th:permutA_m2nd class}, we were able to find only a few sporadic examples of quadratic permutations, giving rise to bent functions outside $\cM^\#$. However, we believe that the proof techniques used for the case $wt(d)=2$ in the proof of Theorem~\ref{th:quadperm}, can be useful for deriving similar results for other classes of quadratic permutations.
				\end{remark}
				
			
				In the following example, we indicate that (due to Theorem~\ref{th:quadperm}) the conditions of 	Theorem~\ref{th: Outside MM} can indeed be relaxed, since the permutations  $\pi_i(y)\mapsto \alpha_i y^{-1}$ that we use in this example are quadratic and therefore do not satisfy the conditions of Theorem~\ref{th: Outside MM}. As already mentioned in Theorem~\ref{th:Pas_Goh}, monomial quadratic permutations have at least one component that admits  non-trivial linear structures and essentially Theorem \ref{th:quadperm} is the first result in the literature that ensures the outside $\cM^\#$ property  using this class of permutations.  
				\begin{example}\label{ex: bent from monomials outside MM 2}
					Let $m = 3$ and the multiplicative group of $\F_{2^{3}}$ be given by $\F_{2^{3}}^*=\langle a \rangle$, where
					the primitive element $a$ satisfies $a^3 + a + 1=0$. Let $d=2^m-2=6$, which satisfies $d^2 \equiv 1 \mod 7$. Define $\alpha_1=a, \alpha_2=a^4,\alpha_3=a^6$ and $\alpha_4=\alpha_1+\alpha_2+\alpha_3=1$.  By Theorem~\ref{th:permutA_m2nd class}, the mappings $\pi_i(y)=\alpha_i y^d$, for $i=1,2,3$, are involutions, as well as $\pi_4=\pi_1+\pi_2+\pi_3$. Set $k=3$ and $\sigma=id$. Then, quadratic Boolean functions $h_i\in\mathcal{B}_3$ are given in the following way: 
					\begin{equation}
						\begin{split}
							h_1(y)&=Tr(\beta_1y^k)=Tr\left(\frac{\alpha_2}{\alpha_1^3} y^3\right)=Tr(a^3 y^3),\\
							h_2(y)&=Tr(\beta_2y^k)=Tr\left(\frac{\alpha_3}{\alpha_2^3} y^3\right)=Tr(a^2 y^3),\\
							h_3(y)&=Tr(\beta_3y^k)=Tr\left(\frac{\alpha_4}{\alpha_3^3} y^3\right)=Tr(a y^3),\\
							h_4(y)&=Tr(\beta_4y^k) + 1=Tr\left(\frac{\alpha_1}{\alpha_4^3} y^3\right) +1=Tr(y^3) + 1.
						\end{split}
					\end{equation}
					Since the Boolean functions $h_i$  satisfy the condition~\eqref{eq: hi condition} of Theorem~\ref{th: Frobenius extended}, that is, 
						$\sum_{i=1}^{4} h_i(\pi_i^{-1}(x))=1$,
					we have that $f=f_1|| f_2|| f_3||f_4\in\mathcal{B}_{8}$ is bent, where  $f_i(x,y)=Tr(x\pi_i(y))+h_i(y)$ for $i=1,2,3,4$, with $x,y\in\F_{2^3}$.  The algebraic normal form of $f\in\mathcal{B}_8$ is given by
					\begin{equation*}
						\begin{split}
							f(z)=&
							z_2 z_4 + z_1 z_5 + z_4 z_5 + z_3 z_4 z_5 + z_6 + z_1 z_6 + z_3 z_6 + z_1 z_4 z_6 + z_2 z_4 z_6 \\ + & z_2 z_5 z_6 + z_4 z_7 + z_1 z_4 z_7 + z_2 z_4 z_7 + z_3 z_4 z_7 + z_5 z_7 + 
							z_2 z_5 z_7 + z_3 z_5 z_7 \\ + & z_1 z_4 z_5 z_7 + z_3 z_4 z_5 z_7 + z_6 z_7 + z_1 z_6 z_7 + z_2 z_6 z_7 + z_1 z_4 z_6 z_7 + z_5 z_6 z_7 \\ + & z_1 z_5 z_6 z_7 + z_2 z_5 z_6 z_7 + 
							z_3 z_5 z_6 z_7 + z_3 z_4 z_8 + z_2 z_5 z_8 + z_1 z_4 z_5 z_8 \\ +& z_2 z_4 z_5 z_8 + z_1 z_6 z_8 + z_2 z_4 z_6 z_8 + z_3 z_4 z_6 z_8 + z_3 z_5 z_6 z_8 + z_7 z_8 + z_4 z_7 z_8 \\ + & 
							z_5 z_7 z_8 + z_6 z_7 z_8 + z_5 z_6 z_7 z_8.
						\end{split}
					\end{equation*}
					Using a computer algebra system, one can check (for example, as described in~\cite{PPKZ2023}) that the obtained bent function $f\in\mathcal{B}_8$ is outside  $\cM^\#\cup\mathcal{PS}^\#$.
				\end{example}	
			\begin{openproblem}
			As we showed in Section~\ref{sec: 2 Secondary constructions}, it is possible to derive new permutations $\phi_i$ of $\F_2^{m+k}$ with the $(\mathcal{A}_{m+k})$ property from permutations $\pi_i$ of $\F_2^{m}$ with the $(\mathcal{A}_{m})$ property, and additionally to derive Boolean functions $h_i'\in\mathcal{B}_{m+k}$ from   $h_i\in\mathcal{B}_{m}$, such that $f'=f'_1||f'_2||f'_3||f'_4\in\mathcal{B}_{2m+2k+2}$ is bent, where $f'_i(x,y)=x\cdot\phi_i(y)+h'_i(y)$ is bent on $\F_2^{m+k}\times\F_2^{m+k}$. 
			Assuming that $f=f_1||f_2||f_3||f_4\in\mathcal{B}_{2m+2}$ is outside $\cM^\#$, where $f_i(x,y)=x\cdot\pi_i(y)+h_i(y)$ is bent on $\F_2^m\times\F_2^m$, does the function  $f'=f'_1||f'_2||f'_3||f'_4$ belong to $\mathcal{M}^\#$?
			\end{openproblem}
				\section{An application to the design of homogeneous cubic bent functions}\label{sec: 5 Hom bent}
				Despite the intensive study of bent functions in the past decades, only a few works are dedicated to the study of the subclass of \textit{homogeneous bent functions} (we refer the reader to the seminal paper~\cite{Charnes2002} as well as recent contributions~\cite{PolujanPhD,Polujan2020,ZPBW23}). In this section, we show how bent functions satisfying the dual bent condition and permutations with the $(\mathcal{A}_m)$ property can be used for the construction of homogeneous cubic bent functions, which (together with quadratic) are the only homogeneous bent functions known to exist.
				\begin{proposition}\label{prop: homogeneous concatenation}
					Let $f_1\in\mathcal{B}_n$ be a homogeneous cubic bent function. Let $q_1,q_2 \in\mathcal{B}_n$ be two homogeneous quadratic functions, such that $f_2=f_1+q_2$ and $f_3=f_1+q_3$ are bent, and additionally $f_1+f_2+f_3$ is also bent. Defining $f_4=f_1+f_2+f_3+s$ for $s\in\mathcal{B}_n$, the function $f=f_1||f_2||f_3||f_4\in\mathcal{B}_{n+2}$ is homogeneous cubic bent if and only if $f_1^*+f_2^*+f_3^*=(f_1+f_2+f_3+s)^*+1$, where $s\in\mathcal{B}_n$ is a linear function.
				\end{proposition}
				\begin{proof}
					Substituting the defined functions $f_i$ in the bent concatenation~\eqref{eq:ANF_4conc}, we get
					\begin{equation}\label{eq: hom proof}
						\begin{split}
							f(z,z_{n+1},z_{n+2})&=f_1(z) +  z_{n+1}(f_1 +  f_3)(z)\\ &+  z_{n+2}(f_1 +  f_2)(z) +  z_{n+1}z_{n+2}(f_1 +  f_2 +  f_3 +  f_4)(z) \\
							&=f_1(z) +  z_{n+1}q_3(z) +  z_{n+2}q_2(z) +  z_{n+1}z_{n+2}s(z).
						\end{split}
					\end{equation}
					Since the first three terms in~\eqref{eq: hom proof} are homogeneous and cubic, the last term $z_{n+1}z_{n+2}s(z)$ must be either homogeneous and cubic, or zero; this is possible if and only if the function $s\in\mathcal{B}_n$ is linear. Since $f_1^*+f_2^*+f_3^*=(f_1+f_2+f_3+s)^*+1$, we have that the dual bent condition $f_1^*+  f_2^* + f_3^*   +  f_4^* =1$ is satisfied, and hence $f\in\mathcal{B}_{n+2}$ is bent. \qed
				\end{proof}
				\begin{example}
					Consider the following homogeneous functions $f_1,q_2,q_3,s\in\mathcal{B}_8$, which are given by their algebraic normal forms as follows:
					\begin{equation*}
						\begin{split}
							f_1(z)= & z_1 z_2 z_5 +  z_1 z_2 z_8 +  z_1 z_3 z_4 +  z_1 z_3 z_5 +  z_1 z_3 z_6 +  z_1 z_3 z_7 +  z_1 z_4 z_5 +  z_1 z_4 z_7 \\ 
							+&  z_1 z_4 z_8 + z_1 z_5 z_8 +  z_1 z_6 z_8 +  z_2 z_3 z_4 +  z_2 z_3 z_5 +  z_2 z_4 z_5 +  z_2 z_4 z_6 +  z_2 z_4 z_8 \\ 
							+&  z_2 z_5 z_6 +  z_2 z_6 z_7 + z_2 z_6 z_8 +  z_2 z_7 z_8 +  z_3 z_4 z_6 +  z_3 z_4 z_8 +  z_3 z_5 z_6 +  z_3 z_5 z_7 \\ 
							+&  z_3 z_6 z_8 +  z_4 z_7 z_8 +  z_5 z_6 z_7 + z_5 z_6 z_8,\\
							q_2(z)= &z_1 z_4 + z_1 z_5 + z_1 z_7 + z_5 z_7 + z_1 z_8 + z_4 z_8 + z_6 z_7 + z_6 z_8 + z_7 z_8, \\
							q_3(z)=&z_1 z_3 + z_1 z_4 + z_1 z_7 + z_1 z_8 + z_2 z_3 + z_2 z_8  + z_3 z_5 + z_3 z_8 + z_4 z_7 \\ 
							+& z_5 z_6 + z_6 z_7 + z_7 z_8,\\
							s(z)=&  z_1 + z_4 + z_6 + z_8.\\
						\end{split}
					\end{equation*}
					One can check that the functions $f_1,q_2,q_3,s\in\mathcal{B}_8$ satisfy the conditions of Proposition~\ref{prop: homogeneous concatenation}, and hence $f=f_1||f_2||f_3||f_4\in\mathcal{B}_{10}$ constructed as in Proposition~\ref{prop: homogeneous concatenation} is homogeneous cubic bent. Notably, there exists a linear non-degenerate transformation  $z\mapsto zA$, where the matrix $A$ is given by
					$$A=\scalebox{0.85}{$\left(
						\begin{array}{cccccccc}
							1 & 0 & 0 & 1 & 0 & 1 & 0 & 1 \\
							0 & 1 & 0 & 1 & 0 & 1 & 1 & 0 \\
							0 & 0 & 1 & 1 & 0 & 0 & 1 & 1 \\
							0 & 0 & 0 & 0 & 1 & 1 & 1 & 1 \\
							0 & 0 & 0 & 1 & 0 & 0 & 0 & 0 \\
							0 & 0 & 0 & 0 & 0 & 1 & 0 & 0 \\
							0 & 0 & 0 & 0 & 0 & 0 & 1 & 0 \\
							0 & 0 & 0 & 0 & 0 & 0 & 0 & 1 \\
						\end{array}
						\right)$}, $$
					such that  $f_i(zA)=x\cdot \pi_i(y)+h_i(y)$, where permutations $\pi_i$ and Boolean functions $h_i$ are defined in Example~\ref{ex:decomposing_h}, and hence, permutations $\pi_i$ have the $(\mathcal{A}_4)$ property. Finally, we note that the function $f\notin\mathcal{M}^\#$ since the functions $f_i$ satisfy the conditions of~\cite[Theorem 5.8]{PPKZ2023}.
				\end{example}
				\begin{openproblem}
					Find explicit infinite families of homogeneous bent functions using the dual bent condition and permutations with the $(\mathcal{A}_m)$ property.
				\end{openproblem}
\section{Conclusions}\label{sec: 6 Conclusion}
In this article, we provided construction methods of permutations of $\F_{2^m}$ with the $(\mathcal{A}_m)$ property as well as their application to the design of bent functions satisfying the dual bent condition. Most notably, concatenating Maiorana-McFarland bent functions satisfying the dual bent condition, we were able to provide a generic construction method of bent functions outside $\mathcal{M}^\#$, including an explicit method based on permutation monomials.

To conclude the paper, we give the following list of problem, solutions to which could provide a better understanding of construction methods of bent functions using the bent 4-concatenation.
\begin{enumerate}
	\item In Proposition~\ref{th: concatenating monomials}, we specified the functions $h_i$ for permutations $\pi_i$ from Theorem~\ref{th:permutA_m2nd class}, such that the dual bent condition for the corresponding Maiorana-McFarland bent functions $f_i(x,y)=T(x\pi_i(y))+h_i(y)$ is fulfilled. However, as computational results indicate, there are much more choices of such functions $h_i$. We think it would be interesting to  provide other construction methods of functions $h_i$, such that the dual bent condition for Maiorana-McFarland bent functions $f_i$ is satisfied.
	\item Similarly to Theorem~\ref{th:quadperm}, provide other explicit classes of permutation polynomials, which can be used for the construction of Maiorana-McFarland bent functions satisfying the dual bent condition. As soon as such classes are identified, show that the concatenation of corresponding Maiorana-McFarland bent functions is outside $\cM^\#$.
	\item As results in Section~\ref{sec: 4} indicate, the strict conditions of Theorem~\ref{th: Outside MM} can be relaxed. In this way, we propose to relax the conditions imposed on the permutations used in Theorem~\ref{th: Outside MM} and provide other generic construction methods of bent functions outside $\cM^\#$ using concatenation of Maiorana-McFarland functions.
\end{enumerate} 

\section*{Acknowledgments} 
Enes Pasalic is supported in part by the Slovenian Research Agency (research program P1-0404 and research projects J1-1694, N1-0159, J1-2451 and J1-4084). Fengrong Zhang is supported in part by the Natural Science Foundation of China (No. 61972400 and No. 62372346), the Fundamental Research Funds for the Central Universities (XJS221503), and the Youth Innovation Team of Shaanxi Universities. \ \\



\end{document}